\documentclass[a4paper,12pt]{article}
\usepackage[T1]{fontenc}
\usepackage[utf8]{inputenc}
\usepackage{amsmath,amssymb,amsthm}
\usepackage{amsrefs}
\usepackage{geometry}

\newtheorem{theorem}{Theorem}
\newtheorem{lemma}[theorem]{Lemma}
\newtheorem{corollary}[theorem]{Corollary}
\newtheorem{proposition}[theorem]{Proposition}
\theoremstyle{definition}
\newtheorem{definition}{Definition}
\newtheorem{hypothesis}{Hypothesis}
\newtheorem{remark}[theorem]{Remark}
\newtheorem{example}[theorem]{Example}

\newcommand{\vd}{\,\mathrm{d}}
\newcommand{\dd}{\mathrm{d}}
\newcommand{\bx}{\mathbf{x}}
\newcommand{\by}{\mathbf{y}}
\newcommand{\bz}{\mathbf{z}}
\newcommand{\RR}{\mathbb{R}}
\newcommand{\NN}{\mathbb{N}}

\newcommand{\SSS}{\mathbb{S}}
\newcommand{\fE}{\mathfrak{E}}

\DeclareMathOperator{\Lip}{Lip}

\begin{document}

\title{Global existence for rough differential
equations under linear growth conditions}

\author{Massimilano Gubinelli\footnote{CEREMADE, Université Paris-Dauphine \& CNRS UMR 7534, 
Place du Maréchal De Lattre De Tassigny,
75775 Paris CEDEX 16, France, 
\texttt{massimilano.gubinelli@ceremade.dauphine.fr}}
\and Antoine Lejay\footnote{Project-team TOSCA, 
(Institut \'Elie Cartan UMR 7502,
Nancy-Universit\'e, CNRS, INRIA), IECN, Campus scientifique, BP 239, 
54506 Vand\oe uvre-l\`es-Nancy CEDEX, France,
\texttt{Antoine.Lejay@iecn.u-nancy.fr}}}

\date{\today}

\maketitle

\begin{abstract}
We prove existence of global solutions for differential equations 
driven by
a geometric rough path under the condition that the vector fields have linear growth. We show by an explicit counter-example that the linear growth condition is not sufficient if the driving rough path is not geometric.  This settle a long-standing open question in the theory of rough paths. So in the geometric setting we recover the usual sufficient condition for differential equation. The proof rely on a simple mapping of the differential equation from the Euclidean space to a manifold to obtain a rough differential equation with bounded coefficients. 
\end{abstract}


\section{Introduction}

Let us consider the controlled differential equation 
\begin{equation}
\label{eq-31}
y_t=y_0+\int_0^t f(y_s)\vd x_s
\end{equation}
for a Lipschitz continuous function $f:\RR\to\RR$, and a smooth path $x:[0,T]\to\RR$.
If $\psi(t)$ is solution to the ordinary differential equation
$\psi'(t)=f(\psi(t))\vd t$, which does not explode in a finite time, then it is well known that 
the solution to \eqref{eq-31} is $y_t=\psi(x_t)$. From this, 
one may deduce that \eqref{eq-31} may be extended to continuous paths $x:[0,T]\to\RR$, 
whatever their regularity. H.~Doss \cite{doss77a} and H.~Sussmann \cite{sussmann78a} have proposed 
this as a simple way to defined stochastic differential equations.
In addition,~\eqref{eq-31} has a \emph{global solution}, \textit{i.e.} a solution that
does not explodes in a finite time.

Using the commutation of the flows, these results may be extended to deal 
with controlled differential equations of type
\begin{equation*}
y_t=y_0+\int_0^t \sum_{i=1}^m f^i(y_s)\vd x_s^i
\end{equation*}
for a family $(f^i)_{i=1}^m$ of \emph{commuting} Lipschitz vector fields 
(\textit{i.e.} when their Lie brackets vanishes), 
and a path $x=(x^1,\dotsc,x^m)$ from $[0,T]$ to $\RR^m$.

For general vector fields without conditions on Lie brackets, the theory
of rough paths introduced by T.~Lyons and his co-authors 
(see for example \cites{friz06c,gubinelli04a,lejay-rp,lejay-rp2,lyons98a,lyons02b,lyons06b}) has provided 
a way to define the solution to 
\begin{equation}
\label{eq-32}
y_t=y_0+\int_0^t f(y_s)\vd x_s
\end{equation}
for paths $x$ with values in a Banach space $U$, and a vector field $f$ 
such that for any $y\in V$,  $f(y)$ is a linear map from $U$ to $V$, $V$ being another Banach space.
In addition to some additional regularity of $f$, the core idea of the theory
is that the driving path $x$ shall be extended as some enhanced path
with values in some non-commutative truncated tensor space. 
The various component of the tensor provide information about iterated integrals of the path $x$ which naturally appears in formal expansions of the solution to~\eqref{eq-32}. This additional data is at the basis of a generalized notion of integral over $x$ which give a extension by continuity of the (Riemman-Lebesgue-Stieljes-Young)  integral formulation of the differential equation. 
The truncation
order of the tensor space depends on $\lfloor p\rfloor$, where $p$ is 
such that this enhanced path is of finite $p$-variation.
Equation~\eqref{eq-32} is then called a \emph{rough differential equation} (RDE).

\medskip
This paper deals with existence of global solutions to the RDE.
For sake of simplicity and clear exposition of the arguments, we prefer to  restrict ourselves to the analysis of the case $2\le p< 3$. This is  the first non-trivial situation appearing when dealing with rough-paths, already presenting most of the features of the general  theory. It is worth stressing that the content of the paper can be generalized at any $p$ without any substantial changes in the arguments.

In all the different approaches to the solution of the RDE~\eqref{eq-32} present in the literature~\cites{lyons98a,feyel, davie05a, gubinelli04a}  the usual condition for global in time existence of solutions
is that both $f$ and $\nabla f$ shall be bounded.
In particular the boundedness on $\nabla f$ is a necessary condition to 
avoid explosion:  A.M. Davie provided a nice counter-example in~\cite{davie05a}.
Yet, by comparing
with the case of a smooth driving path, one may wonder if global 
existence holds for any vector field $f$ of linear growth.

Since the original work of T.~Lyons~\cite{lyons98a} it is clear that if $f$ is a linear vector field then the solution of the RDE is global. In this case indeed it is possible to give an explicit form of this solution as a linear combination of all the iterated integrals of $x$ with coefficients given by tensor powers of the linear map $f$. The bounds provided by Lyons on the growth of the iterated integrals of a rough paths are enough to have norm convergence of the series.

In \cite{friz06c}, P.~Friz and N.~Victoir give a sufficient condition on the vector field 
to have global existence when the driving signal is a geometric rough path.
By this, we mean a rough path 
which is limit of enhanced paths obtained by lifting smooth paths
using their iterated integrals. Their arguments rely on an extension
of the ideas of A.M. Davie~\cite{davie05a}. Essentially, in the case $2<p<3$, along with linear growth they require that  the combination $f \cdot \nabla f$ shall be globally Holder for suitable index.

In \cite{lejay09a}, one of us has provided some conditions for global existence
in a more general context, including non-geometric rough paths.
Excepted for $p=1$, the conditions found cover only sub-linear
growth (depending on $p$) and logarithmic growth (for any $p$) and leave open
the issue of the linear growing $f$.

It is worthwhile to mention that in infinite dimension, requiring bounded
vector fields is a very restrictive condition. In~\cite{gubinelli-tindel}
global solutions  in the context of stochastic partial differential equations
defined by Young-like integrals are proved to exist for fairly general vector
fields which fails to be bounded (in particular they are bounded with respect
to some norm but of linear growth with respect to another norm, situation
peculiar to the infinite dimensional case). It is also shown that linear vector
fields admit global solutions for SPDEs not  covered by standard rough path
theory due to the distributional character of the driving signal.
In~\cite{rooted-ns,kdv} some examples of global solutions to (generalized) RDEs
are provided  in an infinite-dimensional context where the vector field is of
polynomial character. They rely on special features like smallness conditions
on the initial data or on conservation laws.

\medskip
The main result of this paper is the proof  that in the case of geometric rough paths, the linear growth condition on the vector field $f$ is sufficient for global existence. To our knowledge this result answers a natural question which has been a long-standing open issue in rough path theory.
The method of proof relies on mapping the RDE to an RDE with bounded coefficients. Along the way we  prove some change of variable formula for RDEs which can be used as a basis of a general theory of RDEs on manifolds.

 Moreover, for non-geometric rough path, we give a sufficient condition for global existence  related to the behavior of $f\cdot \nabla f$. This last result extends the observation of P.~Friz and N.~Victor to the general non-geometric setting. Our proof is different
and more direct than the one contained in \cite{friz06c}.  By means of an example we show also that there exists non-geometric RDEs with linear growing vector fields which explode in finite time so that some condition on $f\cdot \nabla f$ seems necessary.

It would be interesting to extends these result for general Banach space valued RDEs. This seems not entirely trivial since we exploit some differentiability properties of the norm $|\cdot|$ and of the map $x \mapsto x/|x|$.

Most of the literature on rough paths have been developed in the standard framework of Lyons but recently some papers are written using the language and results associated to the notion of \emph{controlled paths} introduced in~\cite{gubinelli04a}. In this paper we decided to stick to the classical framework enriched with the notion of \emph{partial rough path} (see Sect.~\ref{sec:partial}) which allow a finer analysis of the objects involved in the change of variable formula. However it is worthwhile to note that all our results
can be equivalently stated in terms of controlled paths, change of variable formula for controlled paths and the estimates contained in~\cite{gubinelli04a} on solutions of RDEs in the space of controlled paths.   

\bigskip
\noindent\textbf{Outline.} The paper is organized as follows. Sec.~\ref{sec:prelim} contains some preliminary material in order to fix notation and to define the notion of partial controlled paths and the related RDE estimates. In Sec.~\ref{sec:change} we prove the basic change of variable formula for geometric rough paths and as a by-product a similar formula for non-geometric rough paths. Sec.~\ref{sec:log} will introduce our basic tool for the proof of global existence: a simple logarithmic change of variable which transforms linear growing vector fields to bounded ones. Finally in Sec.~\ref{sec:global} we prove the absence of explosion under linear growth condition when the driving signal is geometric. We give also a new proof of some sufficient condition of Victoir and Friz for non-explosion in the non-geometric setting and we conclude with an explicit example of explosion with linear vector fields in the non-geometric case.

%%%%%%%%%%%%%%%%%%%%%%%%%%%%%%%%%%%%%%%%%%%%%%%%%%%%%%%%%%%%%%%%%%%%%%
\section{Preliminary considerations}
\label{sec:prelim}
%The $\gamma$-Hölder norm 
%of an $\gamma$-Hölder continuous function $f$ from a Banach space $X$
%to a Banach space $Y$ is denoted by $H_\gamma(f)$. 
In the following $U$, $V$ and $W$ will stay for generic Banach spaces.
We denote by $L(U,V)$ the vector space of linear maps from $U$ to $V$.

\begin{definition}
\label{def-1}
For $\gamma\in(0,1]$, a $\Lip(\gamma)$ function from $U$ to $V$
is a map $f:U\to V$ which is $\gamma$-Hölder continuous.
A $\Lip(1+\gamma)$ function from $U$ to $V$ is
$\Lip(1)$ map from $U$ to $V$ such that 
there exists a bounded  $\Lip(\gamma)$  map~$\nabla f$ from $U$ to $L(U,V)$ that satisfies
for any $(u,u')\in U\times U$,
\begin{equation}
\label{eq-26}
\left|f(u)-f(u')-\nabla f(u')(u-u')\right|\\
\leq C|u'-u|^{1+\gamma}.
\end{equation}
The smallest constant $C$ such that \eqref{eq-26} holds for
any $(u,u')\in U\times U$ is denoted by $H_\gamma(\nabla f)$.
\end{definition}

Inequality \eqref{eq-26} is true if $\nabla f$ is $\gamma$-Hölder
continuous.

\begin{remark}
Note that in this definition is a bit different from the one usually used \cites{stein70a,lyons98a,lyons02b},
since we do not impose that $f$ is bounded (but $\nabla f$ is), 
so that $f$ has at most a linear growth.
\end{remark}

\begin{remark}
Later, we may also consider functions $f$ that are defined only on a subset 
$\Omega$ of the Banach space $U$. In this case, by a $\Lip(\gamma)$ functions
we mean a $\gamma$-H\"older continuous function on $\Omega$ and 
by a $\Lip(1+\gamma)$ functions, we mean a function as above for which there exists
$\nabla f$ from $\Omega$ to $L(U,V)$ satisfying~\eqref{eq-26}.
\end{remark}

If $f$ is bounded, then we set 
$\|f\|_{\Lip}=\max\{H_\gamma(\nabla f),\|\nabla f\|_{\infty},\|f\|_{\infty}\}$, 
which is its \emph{Lipschitz norm}.

Given a vector space $U$ we let $T(U)=\oplus_{k\ge 0} U^{\otimes k}$ be the
tensor algebra of $U$ (with $U^{\otimes 0}=\RR$) and $T_2(U)=\RR\oplus U\oplus (U\otimes U)$ the
projection on the elements of degree smaller or equal than $2$ which is again an algebra for the
tensor product. We denote by $\pi_{U^{\otimes k}}:T(U)\to U^{\otimes k}$ the
projection on the subspace of degree $k$.  The space $T_2(U)$ will be equipped
with a compatible norm so that it become a Banach algebra.
 
Let $\bx$ be a rough path of finite $p$-variation controlled by $\omega$ with $p\in[2,3)$
with values in~$T_2(U)$. That is $(\bx_t)_{t\in[0,T]}$ is a path
with values in the subset $\{1\}\oplus U\oplus (U\otimes U)$ of $T_2(U)$ ---~which is
a Lie group for the tensor product when keeping only the terms of degree $0$, $1$ or $2$~--- such that $\bx_{s,t}=\bx_s^{-1}\otimes\bx_t$
satisfies $|\pi_U(\bx_{s,t})|\leq C\omega(s,t)^{1/p}$ and 
$|\pi_{U\otimes U}(\bx_{s,t})|\leq C\omega(s,t)^{2/p}$ for some constant $C$
and any $0\leq s\leq t\leq T$. 

The theory of rough paths has been the subject of several books and lecture notes
\cites{friz06c,gubinelli04a,lejay-rp,lejay-rp2,lyons98a,lyons02b,lyons06b}, 
so that we do not give here more insights of this theory.

However, let us recall here the most technical core of the theory, which 
we call the sewing Lemma, following \cite{sewing-lemma}.

\begin{lemma}[Sewing Lemma]
\label{lem-sewing}
An almost rough path $\tilde\bz$ on $U$ is a map $\tilde \bz: [0,T]^2\to T_2(U)$ such that, for some $\theta>1$,  $|\tilde \bz_{s,t}-\tilde \bz_{s,u} \otimes \tilde \bz_{u,t}|\le C \omega(t,s)^{\theta}$ for all $0\le s\le u \le t\le T$. Given an almost rough path on $U$  then there exists only one rough path $\bz= \mathcal{J}(\tilde \bz)$ on $U$ such that $\bz_0=1$ and $|\bz_{s,t} - \tilde \bz_{s,t}|\le C'\omega(t,s)^{\theta}$.
\end{lemma}

Let $f$ be a $\Lip(1+\gamma)$-map from $V$ to $L(U,V)$ and  
let us consider the RDE
\begin{equation}
\label{eq-rde}
\by_t=a+\int_0^t f(y_s)\vd \bx_s
\end{equation}
for $a\in V$.
By this, we mean that $\by$ is a rough path of finite $p$-variation controlled
by $\omega$ such that $\by$ lives in $T_2(U\oplus V)$, such that $\bx = \pi_{T_2(U)}(\by)$ and which satisfies the equality
\begin{equation*}
\by_t=a+\int_0^t \widehat{f}(y_s)\vd \by_s
\end{equation*}
for the $\Lip(1+\gamma)$-differentiable form from $U\oplus V$ to $U\oplus V$
defined by $\widehat{f}(v)(u',v')=u'+f(v)u'$.
We used the convention that the italic letter 
$y$ is path in~$V$ obtained by 
the projection of the rough path denoted by a bold letter $\by$ (i.e. $y= \pi_V(\by)$).

Another equivalent characterization of $\by$ is given by the fact that it is the unique rough path satisfying $\by = a + \mathcal{J}(\tilde \by)$ where $\tilde \by$ is the almost rough path
\begin{equation*}
\begin{split}
\tilde \by_{s,t} & =  \, \bx_{s,t}+f(y_s) \pi_U(\bx_{s,t}) + (f\cdot\nabla f)(y_s)\pi_{U\otimes U}(\bx_{s,t})
\\ & \quad + f(y_s) \otimes f(y_s)\pi_{U\otimes U}(\bx_{s,t}) .
\end{split}
\end{equation*} 

We know that $\by$ exists at least up to some explosion time $T$, 
which is characterized by $\lim_{t\to T-}|\by_t|=+\infty$
for $T<+\infty$ (See for example~\cite{lejay09a}).
Besides, if $f$ is bounded, then no explosion occurs and in 
this case, we say that there exists a \emph{global solution to \eqref{eq-rde}}.
Uniqueness (and continuity of the map $\bx\mapsto\by$ in the $p$-variation topology)
is granted only if $f$ is a $\Lip(2+\gamma)$-vector field from $V$ to $L(U,V)$
(see \cite{davie05a} for a counter-example).

If $x$ is a smooth path, then one may solve first
the ODE
\begin{equation}
\label{eq-ode}
y_t=a+\int_0^t f(y_s)\vd x_s
\end{equation}
in $V$ and then construct $\bx$ and $\by$ as (in the following expression, 
the elements to be summed live in different spaces)
\begin{align*}
\bx_t&=1_{T(U)}+x_t+\int_0^t (x_s-x_0)\otimes\dd x_s,\\
\by_t&=1_{T(V)}+\bx_t+y_t+\int_0^t (y_s-y_0)\otimes\dd x_s
+\int_0^t (y_s-y_0)\otimes\dd y_s\\
& \qquad
+\int_0^t (x_s-x_0)\otimes\dd y_s.
\end{align*}
In this case, $\bx$ and $\by$ are \emph{smooth rough paths}.

\subsection{Transformation of a partial rough path by a smooth function}
\label{sec:partial}
Let $x$ and $y$ two paths of finite $p$-variation respectively
with values in $U$ and~$V$, and such that the cross-iterated
integral $\int \dd y\otimes\dd x$ exists and is controlled by~$\omega$. This means
that $\int_s^t \dd y\otimes \dd x$ lives in $V\otimes U$ and satisfies 
for some constant~$L$,
\begin{gather}
\label{eq-24}
\left|\int_s^t \dd y\otimes\dd x\right|\leq L\,\omega(s,t)^{p/2},\\
\label{eq-25}
\int_s^t \dd y\otimes\dd x=\int_s^r \dd y\otimes\dd x+\int_r^t \dd y\otimes \dd x
+(y_r-y_s)\otimes (x_t-x_r),
\end{gather}
for all $0\leq s\leq r\leq t\leq T$.

Note that knowing a cross-iterated integral between $y$ and $x$ is sufficient to properly define an 
integral of type $\int f(y_s)\vd x_s$ with value in a Banach space $W$
for a $\Lip(1+\gamma)$-vector field from $V$ to $L(U,W)$,
by defining this integral as the element associated
to the almost rough path $f(y_s)x_{s,t}+\nabla f(y_s)\int_s^t \dd y\otimes\dd x$
using the non-commutative Sewing Lemma~\ref{lem-sewing} while keeping only the 
elements in $U\oplus V\oplus (V\otimes U)$.

We then call the triple $\left(x,y,\int \dd y\otimes\dd x\right)$
a \emph{partial rough path}. The distance between two partial rough paths 
$\left(x,y,\int \dd y\otimes\dd x\right)$ and 
$\left(x',y',\int \dd y'\otimes\dd x'\right)$ 
is given by 
\begin{equation*}
\sup_{0\leq s<t\leq T}
\max\left\{\frac{|x_{s,t}-x'_{s,t}|}{\omega(s,t)^{1/p}},
\frac{|y_{s,t}-y'_{s,t}|}{\omega(s,t)^{1/p}},
\frac{\left|\int_s^t \dd y\otimes\vd x-\int_s^t \dd y'\otimes\vd x'\right|}{\omega(s,t)^{p/2}}\right\}.
\end{equation*}
The corresponding topology is called the \emph{topology of $p$-variation generated by $\omega$}.

\begin{lemma}
\label{lem-3}
Let $\phi$ be a $\Lip(1+\kappa)$ map from $V$ to $W$, $\kappa>p$.
Then there exists a cross-iterated integral for $\phi(y)$ 
and $x$ (that is a function that satisfies \eqref{eq-24}
and \eqref{eq-25} with $y$ replaced by $\phi(y)$) which 
extend the smooth  natural iterated integral when $x$ and $y$ are smooth.
In addition, the map $\left(y,x,\int \dd y\otimes\dd x\right)
\mapsto \left(\phi(y),x,\int \dd\phi(y)\otimes\dd x\right)$
is continuous with respect to the topology of $p$-variation generated by $\omega$.
\end{lemma}

With the above hypothesis on $\phi$, 
$\phi(y)$ is of finite $p$-variation controlled by~$\omega$
on $V$, and 
this last lemma implies of course that $\int g(\phi(y_s))\vd x_s$
is well defined for a $\Lip(1+\gamma)$-vector field $g$ from $W$ to $L(U,Y)$ for a Banach space~$Y$.

\begin{proof}
Set
\begin{equation*}
z_{s,t}=\nabla \phi(y_s)\int_s^t \dd y\otimes\dd x+\phi(y_t)-\phi(y_s)
+x_t-x_s
\end{equation*}
living in $U\oplus V\oplus (V\otimes U)$. With abuse of notation here we note $\nabla \phi(y_s)$ the application $\nabla \phi(y_s) \otimes 1 \in L(V,W) \otimes L(U,U)\simeq L(V \otimes U, W \otimes U)$. The context will always be sufficient to remove the ambiguity.

It follows that, by keeping only the terms in $U\oplus V\oplus (V\otimes U)$,
\begin{multline*}
z_{s,t}-z_{s,r}\otimes z_{r,t}
=\nabla \phi(y_s)(y_r-y_s)\otimes (x_r-x_s)\\
+(\phi(y_r)-\phi(y_s))\otimes (x_r-x_s)
+(\nabla\phi(y_s)-\nabla \phi(y_r))\int_r^t \dd y\otimes\dd x.
\end{multline*}
In addition
\begin{equation*}
\phi(y_r)-\phi(y_s)=\int_0^1 \nabla \phi(y_s+\tau(y_r-y_s))(y_r-y_s)\vd \tau.
\end{equation*}
It follows from standard estimates that 
\begin{equation*}
|z_{s,t}-z_{s,r}\otimes z_{r,t}|\leq C\,\omega(s,t)^{(1+\kappa)/p}
\end{equation*}
where $C$ depends only on $\|\nabla \phi\|_{\infty}$, $H_\gamma(\nabla \phi)$, 
$\|y\|_{p,\omega}$, $\|x\|_{p,\omega}$, $L$ in \eqref{eq-24} and~$\omega(0,T)$.
Consequently, $z$ is an almost rough path. Applying the sewing map $\mathcal{J}$ of Lemma~\ref{lem-sewing}
to $z$ while keeping only the terms in $U\oplus V\oplus (V\otimes U)$, we get an triple
$\left(x,\phi(y),\int\dd\phi(y)\otimes\dd x\right)$, where 
$\int\dd\phi(y)\otimes\dd x$ is a cross-iterated integral between $\phi(y)$ and $x$.

With similar computations, it is easily shown that 
$\left(x,y,\int \dd y\otimes\dd x\right)\mapsto \int \dd\phi(y)\otimes\dd x$ is continuous.

If $x$ and $y$ are smooth, 
then 
\begin{equation*}
\begin{split}
& \left|\int_s^t (\phi(y_r)-\phi(y_s))\vd x_r
-\nabla \phi(y_s)\int_s^t \dd y\otimes\dd x
\right|\\
& \qquad \leq \left|\int_s^t \left(\int_0^1 (\nabla \phi(y_s+\tau(y_r-y_s))-\nabla\phi(y_s)) \vd \tau (y_r-y_s)\right) \otimes\dd x_r
\right|\\
& \qquad \leq C\omega(s,t)^{(2+\kappa)/p}.
\end{split}
\end{equation*}
which means that $\int_0^t \dd\phi(y)\otimes\dd x$ is the rough 
path associated to the almost rough path 
$\nabla \phi(y_s)\int_s^t \dd y\otimes\dd x$.
\end{proof}

%%%%%%%%%%%%%%%%%%%%%%%%%%%%%%%%%%%%%%%%%%%%%%%%%%%%%%%%%%%%%%%%%%%%%%
\subsection{A bound on the solutions of RDE  for bounded vector fields}

Let $h$ a bounded $\Lip(1+\gamma)$-vector field from $V$ to $L(U,V)$, 
$\gamma\in(0,1]$, $2+\gamma>p$.
By this, we mean that $h$ satisfies~\eqref{eq-26}.
Consider the RDE in $V$
\begin{equation}
\label{eq-19}
\bz_t=a+\int_0^t h(z_s)\vd \bx_s
\end{equation}

\begin{proposition}
\label{prop-3}
Under the above hypotheses, 
there exists a constant $C$ that depends only on $\|h\|_{\Lip}$, 
$\gamma$ and $p$ such that 
\begin{equation}
\label{eq-18}
\sup_{t\in[0,T]}|z_t-z_0|\leq C(1+\|\bx\|_{p,\omega}^p\omega(0,T)).
\end{equation}
\end{proposition}

\begin{proof}
From the computations in \cite{lejay09a}, for any $T>0$
and some constant $\mu$ depending on $\|\nabla h\|_{\infty}$,
$H_\gamma(\nabla h)$ and $\gamma$,
there exists for any time $s$ a time $s'$ such that 
$|z_t-z_s|\leq \mu$ and 
\begin{equation}
\label{eq-17}
\omega(s,s')^{1/p}(\|h\|_{\infty}
+\mu\|\nabla h\|_{\infty})\|\bx\|_{p,\omega}\leq K\mu
\end{equation}
for some universal constant $\mu$. If $\omega$ is continuous, 
then we may choose $s'$ so that an equality holds in \eqref{eq-17}. 
Hence, we construct recursively a family $(s_n)_{n\geq 1}$ of times 
such that $\omega(s_{n},s_{n+1})=L\|\bx\|_{p,\omega}^{-p}$, 
where $L$ depends only on the norm $\|h\|_{\Lip}$. 
Let $N$ be the smallest time such that $s_{N+1}\geq T$.
Since $\sum_{i=1}^{N-1} \omega(s_i,s_{i+1})\leq \omega(0,T)\leq
\sum_{i=1}^{N} \omega(s_i,s_{i+1})$, we have 
\begin{equation*}
(N-1) L\|\bx\|^{-p}_{p,\omega}\leq \omega(0,T)
\leq N L\|\bx\|^{-p}_{p,\omega}
\end{equation*}
and 
\begin{equation}
\sup_{t\in[0,T]}|z_t-z_0|\leq N\mu \leq \mu+L'\omega(0,T)\|\bx\|_{p,\omega}^p
\end{equation}
with $L'=\mu/L$, which is a constant that depends only on 
 $\|h\|_{\infty}$, $\|\nabla h\|_{\infty}$, $H_\gamma(h)$,
 $\gamma$ and $p$.
\end{proof}

%%%%%%%%%%%%%%%%%%%%%%%%%%%%%%%%%%%%%%%%%%%%%%%%%%%%%%%%%%%%%%%%%%%%%%
%%%%%%%%%%%%%%%%%%%%%%%%%%%%%%%%%%%%%%%%%%%%%%%%%%%%%%%%%%%%%%%%%%%%%%
\section{Change of variable formulas}
\label{sec:change}
%%%%%%%%%%%%%%%%%%%%%%%%%%%%%%%%%%%%%%%%%%%%%%%%%%%%%%%%%%%%%%%%%%%%%%
\subsection{Geometric rough paths}

Assume $y$ is a solution to some RDE,
we prove here a change of variable formula giving the RDE satisfied by $\phi(y_t)$. We need the following 
regularity hypothesis on~$\phi$.

\begin{hypothesis}
\label{hyp-2}
For $\gamma\in(0,1]$, 
the function $\phi$ is a $\Lip(1+\gamma)$ function from~$V$ to~$W$ 
which is one-to-one between a set $\Gamma\subseteq V$ and a closed set $\Omega\subseteq W$.
\end{hypothesis}

Let $x$ be a smooth path and let $y$ be the solution to \eqref{eq-rde}. 
We assume that $y_t$ belongs to $\Gamma$ for any $t\in[0,T]$.
Then the Newton formula applied to $z_t=\phi(y_t)$ reads
\begin{equation*}
z_t=\phi(y_t)=\phi(a)+\int_0^t \nabla \phi(y_s)f(y_s)\vd x_s
=\phi(a)+\int_0^t  h(z_s)\vd x_s
\end{equation*} 
with, for $z$ in $\Omega$, 
\begin{equation}
\label{eq-14}
h(z)=\nabla \phi(\phi^{-1}(z))f(\phi^{-1}(z)).
\end{equation}

Note that $h$ is defined only on a subspace $\Omega$ of 
a vector space, so that one cannot necessarily 
solve the RDE $\bz_t=\phi(a)+\int_0^t h(z_s)\vd \bx_s$, 
because it involves the derivative of $h$ in directions
that are not necessarily in $\Omega$. In addition, the proofs of existence and 
continuity of solutions of RDE rely on expressions on type
$\int_0^1 \nabla h(z_s+\tau(z_t-z_s))(z_t-z_s)\vd \tau$,
so that $\nabla h$ needs to be defined at least on a convex set containing $\Omega$.
This is why we assume
the following hypotheses.

\begin{hypothesis}
\label{hyp-1}
The function $h$ can be extended to a $\Lip(1+\gamma)$-vector field from $W$ to $L(U,W)$.
\end{hypothesis}

Note that, for Euclidean vector spaces, the Whitney extension 
theorem (see Theorem~\ref{thm-1} below) asserts the existence of such an extension.

Using Lemma~\ref{lem-3}, we see that 
$\int h(z_s)\vd x_s=\int h(\phi(y_s))\vd x_s$
is well defined provided that $\left(x,y,\int \dd y \otimes\dd x\right)$
is a partial rough path, since $h$ is a $\Lip(1+\gamma)$-vector field from $W$ to $L(U,W)$
and $\left(x,\phi(y),\int \dd \phi(y)\otimes\dd x\right)$ is a partial rough path.

If a solution $\by$ to \eqref{eq-rde} exists,
then 
$\left(x,y,\int \dd y\otimes\dd x\right)$ is a 
partial rough path of finite $p$-variation controlled
by $\omega$. In addition if $\bx$ is a geometric rough 
path, then there exists a family $(\bx^n)_{n\in\NN}$
of smooth rough paths that converges to $\bx$ in $p$-variation.
Let $\by^n$ be the solution to the RDE
$\by^n_t=a+\int_0^t f(y^n_s)\vd\bx^n_s$. 
Then one can extract from $\by^n$ the partial 
rough path $\left(x^n,y^n,\int \dd y^n\otimes\dd x^n\right)$
which converges to the partial rough path 
$\left(x,y,\int \dd y\otimes\dd x\right)$.

Letting $z^n=\phi(y^n)$ we have
$z^n_t = \phi(a)+\int_0^t h(z^n_s) dx_s$ and by the convergence of the partial rough paths we have that 
$
\int h(z^n_s) dx_s \to \int h(z_s) dx_s
$.
So we have then proved the following result.

\begin{lemma}
\label{lem-4}
Let $\bx$ be a rough path of finite $p$-variation, 
$\by$ be a solution to \eqref{eq-rde} which we assume to exists up
to time $T$. Let $\phi$ such that Hypotheses~\ref{hyp-2} and~\ref{hyp-1} hold.
Then $z=\phi(y)$ is the solution to 
\begin{equation}
\label{eq-27}
z_t=\phi(a)+\int_0^t h(z_s)\vd x_s
\end{equation}
for $t\leq T$,
where this definition involves the partial rough path
$\left(x,z,\int \dd z\otimes\dd x\right)$ constructed
in Lemma~\ref{lem-4}.
\end{lemma}

Provided that $h$ is bounded, it follows directly from Proposition~\ref{prop-3} 
that one gets a bound on $\sup_{t\in[0,T]}|z_t-z_0|$.
The next proposition is then immediate.

\begin{proposition}
\label{prop-1}
Under the above hypotheses, if $h$ in bounded in $W$, 
then there exists a global solution to the RDE \eqref{eq-rde}.
\end{proposition}

Indeed, if an explosion occurs to \eqref{eq-27} at
time $T$, then an explosion should also occurs to \eqref{eq-rde}
at time $T$.

%%%%%%%%%%%%%%%%%%%%%%%%%%%%%%%%%%%%%%%%%%%%%%%%%%%%%%%%%%%%%%%%%%%%%%

\subsection{Non-geometric rough paths}

Let us consider  a rough path $\bx$ of finite $p$-variation controlled
by $\omega$ with value in a
 Banach space $U$, but which is not necessarily a geometric one.  In~\cite{lejay-victoir06a} it is shown that $\bx$
may be decomposed by $\bx=\widehat{\bx}+\beta$, 
where $\widehat{\bx}$ is a geometric rough path and 
$\beta$ is a path of $p/2$-finite variation with 
values in $U\otimes U$ such that $\beta_t$ is symmetric.
In addition, the map $\bx\mapsto (\widehat{\bx},\beta)$ is continuous.

Besides, there exists a sequence $(\widehat{\bx}^n,\beta^n)$ of approximations
of $(\widehat{\bx},\beta)$, in the sense that $\widehat{\bx}^n$
converges to $\widehat{\bx}$ for the norm of $p$-variation induced by $\omega$
and $\beta^n$ converges to $\beta$ for the norm of $p/2$-variation
induced by $\omega$.

\begin{hypothesis}
\label{hyp-3}
Let $f$ be a $\Lip(1+\gamma)$-vector field from $V$ to $L(U,V)$, $\gamma\in(0,1]$, $2+\gamma>p$ 
such 
that $f\cdot\nabla f$ defined by 
\begin{equation*}
(f\cdot\nabla f)(v)u\otimes w=\nabla f(v)((f(v) u)\otimes w),
\ (v,u,w)\in V\times U\times U,
\end{equation*}
is $\Lip(\gamma)$-vector field from $V$ to $L(U\otimes U,V)$.
\end{hypothesis}

Let $\by$ be the solution to $\by_t=y_0+\int_0^t f(y_s)\vd\bx_s$.
In this case, since $2+\gamma>p$, 
\begin{equation*}
\alpha_t=\int_0^t (f\cdot\nabla f)(y_s)\vd \beta_s
\end{equation*}
is well defined as a Young integral and is of $p/2$-finite variation
controlled by~$\omega$.

Let us denote by $\int \dd y\otimes\dd x$ the cross-iterated integral between $y$
and~$x$. Since 
$\int \dd y\otimes\dd \beta$ is well defined as a Young integral
and $\int \dd y^n\otimes\dd \beta^n$ converges to $\int \dd y\otimes\dd \beta$
as well as  $\int \dd y^n\otimes\dd x^n$ converges to $\int \dd y\otimes\dd x$, an approximation argument 
shows that one may naturally define a cross-iterated integral $\int \dd y\otimes\dd \widehat{x}$ by the formula
\begin{equation}
\label{eq-23}
\int \dd y\otimes\dd \widehat{x}=\int \dd y\otimes\dd x+\int \dd y\otimes\dd \beta.
\end{equation}
It follows that the rough integral $\int_0^t f(y_s)\vd \widehat{x}_s$
is well defined (since one needs only to get 
the iterated integral between $y$ and $\widehat{x}$)
and using an approximation argument, one gets that 
\begin{equation*}
\by_t=a+\int_0^t f(y_s)\vd \widehat{\bx}_s
+\int_0^t (f\cdot \nabla f)(y_s)\vd \beta_s
+\int_0^t f(y_s)\otimes f(y_s)\vd \beta_s.
\end{equation*}
In particular, in using the cross-iterated integral
$\int \dd y\vd\widehat{x}$ given by \eqref{eq-23}, the projection $y_t$
of the rough path $\by_t$ is given by 
\begin{equation*}
y_t=y_0+\pi_V\left(\int_0^t f(y_s)\vd \widehat{\bx}_s\right)+\int_0^t (f\cdot\nabla f)(y_s)\vd \beta_s,
\end{equation*}
where $\pi_V$ is the projection operator from $T_2(V)$ onto $V$.

Now, let us consider a family of approximations
$(\widehat{\bx}^n,\beta^n)$ of $(\widehat{x},\beta)$,
where~$\widehat{\bx}^n$ is a smooth rough path
and $\beta^n$ is a smooth path.
The solution 
\begin{equation*}
y^n_t=a+\int_0^t f(y^n_s)\vd \widehat{x}^n_s+\int_0^t (f\cdot\nabla f)(y^n_s)\vd \beta^n_s
\end{equation*}
is then the projection onto $V$ of the rough solution to 
$\by^n_t=a+\int_0^t f(y^n_s)\vd \bx^n_s$.

\begin{hypothesis}
\label{hyp-4}
We consider a  one-to-one $\Lip(1+\gamma)$ map $\phi$ from $\Gamma\subseteq V$ to a closed subset $\Omega\subseteq W$ such that 
\begin{equation}
\label{eq-28}
h_1(z)=\nabla \phi\circ \phi^{-1}(z)f\circ \phi^{-1}(z)
\end{equation}
may be extended to a $\Lip(1+\gamma)$-vector field from $W$ to $L(U,W)$  and
\begin{equation}
\label{eq-29}
h_2(z)=\nabla \phi\circ \phi^{-1}(z)[f\cdot \nabla f](\phi^{-1}(z))
\end{equation}
may be extended  to a $\Lip(\gamma)$-vector field from $W$ to $L(U,W)$.
We also assume that any solution $y^n$ satisfies $y^n_t\in\Gamma$, $t\in[0,T]$, 
for any $n$.
\end{hypothesis}

Set 
$z^n_t=\phi(y^n_t)$. It follows from the change of variable formula
that
\begin{equation*}
z^n_t=z^n_0+\int_0^t \nabla \phi(y_s^n)f(y_s^n)\vd\widehat{x}^n_s
+\int_0^t \nabla \phi(y_s^n)(f\cdot\nabla f)(y_s^n)\vd \beta_s^n.
\end{equation*}
and then that 
\begin{equation*}
z^n_t=z^n_0+\int_0^t h_1(z^n_s)\vd\widehat{x}^n_s
+\int_0^t h_2(z^n_s)\vd \beta^n_s.
\end{equation*}

Passing to the limit with the help of Lemma~\ref{lem-3}
(let us recall that the definition of $\int h(z_s)\vd x_s$ as
a rough integral requires only to know the cross-iterated integral
between $z$ and $x$ and well as the iterated integrals of $x$),
one gets the following Lemma. 

\begin{lemma}
Under Hypotheses~\ref{hyp-3} and \ref{hyp-4}, 
$z_t=\phi(y_t)$ is solution to 
\begin{equation}
\label{eq-21}
z_t=z_0+\int_0^t h_1(z_s)\vd\widehat{x}_s
+\int_0^t h_2(z_s)\vd \beta_s
\end{equation}
where the integral $\int h_1(z_s)\vd \widehat{x}_s$
is defined using the cross-iterated integral $\int \dd\phi(y_s)\vd \widehat{x}_s$.
\end{lemma}

Let us note that \eqref{eq-21} is different from 
\begin{equation*}
z_t=z_0+\int_0^t h_1(z_s)\vd x_s=z_0+\int_0^t h_1(z_s)\vd \widehat{x}_s
+\int_0^t (h_1\nabla h_1)(z_s)\vd \beta_s.
\end{equation*}

Note also that unless $f$ is bounded, $(f\cdot\nabla f)$ 
may be only \emph{locally} $\gamma$-Hölder continuous, 
so that the boundedness of $h_1$ and $h_2$ is not necessarily 
sufficient to deduce the existence of a global solution 
to \eqref{eq-rde} if one drops Hypothesis~\ref{hyp-3}.

%%%%%%%%%%%%%%%%%%%%%%%%%%%%%%%%%%%%%%%%%%%%%%%%%%%%%%%%%%%%%%%%%%%%%%

%\section{Dealing with the starting point}
%
%Now, let us consider 
%\begin{equation*}
%\by_t=a+\int_0^t f(y_s)\vd \bx_s.
%\end{equation*}
%for a bounded $\Lip(1+\gamma)$-vector field $f$, 
%so that there exists a global solution to this RDE.
%Let us set $g(z)=f(z+a)$, which is also a bounded
%$\Lip(1+\gamma)$-vector field. In addition, 
%$\|g\|_{\Lip}=\|f\|_{\Lip}$.
%
%Using the almost rough paths associated to rough integrals, 
%it is easily checked that 
%\begin{equation*}
%\by_t-a=\int_0^t  f(y_s)\vd \bx_s=\int_0^t g(y_s-a)\vd \bx_s
%\end{equation*}
%so that $\bz_t=\by_t-a$ is a solution to the RDE
%$\bz_t=\int_0^t g(z_s)\vd \bx_s$ (the iterated integrals of $y$
%are not changed by a shift by a constant).
%
%Since in \eqref{eq-18}, $L'$ depends only on $\|f\|_{\Lip}$, 
%it follows that it is always possible for some $T>0$ and some $\delta>0$
%to change the starting point and to 
%to ``shift'' the 
%vector field $f$ in order to assume that $|y_t|\geq \delta>0$ for all $t\in[0,T]$, 

%%%%%%%%%%%%%%%%%%%%%%%%%%%%%%%%%%%%%%%%%%%%%%%%%%%%%%%%%%%%%%%%%%%%%
\section{A convenient transformation of vector fields}
\label{sec:log}
From now, we denote by $\SSS^d$ the sphere
of radius $1$ in $\RR^d$.
%\note{Can we work in infinite dimensions? It seems possible to me. The change of variable only involves the norm of the vector}

Let us consider the one-to-one map $\phi$ from $\RR^d\backslash\{0\}$ to $\Omega=\SSS^d\times \RR_+$
defined by 
\begin{equation*}
\phi(z)=\begin{bmatrix}
\theta(z)\\
\rho(z)
\end{bmatrix}
\text{ with }\theta(z)=\dfrac{z}{|z|}
\text{ and }\rho(z)=\log(|z|)
\end{equation*}
for $z\in\RR^d\backslash\{0\}$. 

\begin{remark}
In this section, $C$ and $C'$ denote constants that may vary from line to line.
\end{remark}

Let us also set the inverse map
\begin{equation*}
z(\theta,\rho)=\exp(\rho)\theta.
\end{equation*}
Since $|\theta|=|\theta'|=1$, for all $(\theta,\rho)\in\Omega$, $(\theta',\rho')\in\Omega$,
\begin{equation}
\label{eq-5}
|z(\theta,\rho)-z(\theta',\rho')|\leq 
|\theta-\theta'|\exp(\rho')+|\exp(\rho)-\exp(\rho')|.
\end{equation}
If $\rho'>\rho$, we have
\begin{equation*}
\exp(\rho')-\exp(\rho)=\int_0^1 \exp(\rho+\tau(\rho'-\rho))(\rho'-\rho)d\tau
\leq \exp(\rho')(\rho'-\rho)
\end{equation*}
so that we transform \eqref{eq-5} as 
\begin{equation}
\label{eq-6}
|z(\theta,\rho)-z(\theta',\rho')|\leq 
\exp(\rho')\big(|\theta-\theta'|+(\rho'-\rho)\big)\text{ when }
\rho'>\rho.
\end{equation}
We have also for $\gamma\in(0,1)$ and $\rho'\geq \rho$,  
\begin{multline}
\label{eq-7}
|z(\theta,\rho)-z(\theta',\rho')|\\
\leq 
\max\{|z(\theta,\rho)|,|z(\theta',\rho')|\}^{1-\gamma}
2^{\gamma-1}\exp(\gamma\rho')(
|\rho'-\rho|^\gamma+|\theta'-\theta|^\gamma)\\
\leq 2^{\gamma-1}\exp(\rho')(|\rho'-\rho|^\gamma+|\theta'-\theta|^\gamma).
\end{multline}
This way, $z(\theta,\rho)$ is locally $\gamma$-Hölder
for any $\gamma\in(0,1]$.

In addition, if $\theta_i(z)=z_i/|z|$, 
\begin{equation*}
\frac{\partial \theta_i(z)}{\partial z_j}=\delta_{i,j}\frac{1}{|z|}
-\frac{z_j^2}{|z|^3}\text{ and }
%\frac{\partial^2 \theta_i(z)}{\partial z_j\partial z_k}
%&=-\delta_{i,j}\frac{z_k}{|z|^3}+\frac{3z_j^2z_k}{|z|^5},\\
\frac{\partial \rho(z)}{\partial z_i}=\frac{z_i}{|z|^2}.
%\frac{\partial^2 \rho(z)}{\partial z_i\partial z_j}&%\frac{\delta_{i,k}}{(1+|z|)|z|}-\frac{
%z_k-z_k|z|}{(1+|z|)^2|z|^2}
\end{equation*}
Computing the first order derivatives for $k=1,2,3$, one gets that
\begin{equation*}
|\nabla^k \theta(z)|\leq C/|z|^k,\ 
|\nabla^k \rho(z)|\leq C'/|z|^k.
\end{equation*}
Hence
\begin{multline}
\label{eq-12}
|\nabla \phi(z)-\nabla \phi(z)|
\leq C
\int_0^1 |z'-z|\cdot |\nabla^{2}\phi(z+\tau(z'-z))|d\tau\\
\leq C\int_0^1 \frac{|z'-z|}{|z+\tau(z'-z)|^2}d\tau
\leq \dfrac{C}{|z|}-\dfrac{C'}{|z'|}\leq \dfrac{C|z'-z|}{|z'|\cdot|z|}.
\end{multline}
This way, 
\begin{equation}
\label{eq-8}
|\nabla \phi(z(\theta,\rho))-\nabla \phi(z(\theta',\rho'))|
\leq \frac{|\rho'-\rho|}{\exp(\rho)}.
\end{equation}
In addition, if $\rho',\rho\geq 1$, 
\begin{multline}
\label{eq-9}
|\nabla \phi(z(\theta,\rho))-\nabla \phi(z(\theta',\rho'))|\\
\leq \frac{|\rho'-\rho|^\gamma+|\theta'-\theta|^\gamma}{\gamma}\max\{|\nabla\phi(z(\theta,\rho))|^{1-\gamma},
|\nabla\phi(z(\theta',\rho'))|^{1-\gamma}\}\\
\leq C \frac{|\rho'-\rho|^\gamma+|\theta'-\theta|^\gamma}{\exp(\rho)}.
\end{multline}

If $f$ is a $\Lip(\gamma)$ vector field from $\RR^d$ to $L(\RR^m,\RR^d)$, then
set 
\begin{equation*}
h(\theta,\rho)=\nabla \phi(z(\theta,\rho))f(z(\theta,\rho))
\end{equation*}
which is a short-hand for
\begin{equation}
\label{eq-30}
h^i_k(\theta,\rho)=\sum_{j=1}^m \frac{\partial \phi^i(z(\theta,\rho))}{\partial z_j}f^j_k(z(\theta,\rho)).
\end{equation}

\begin{lemma}
\label{lem-1}
For $\gamma\in(0,1]$, $h$ is also a $\Lip(\gamma)$-vector field from $\Omega$ to $L(\RR^m,\RR^d)$
and is bounded.
\end{lemma}

\begin{proof}
For $\rho,\rho'\geq 1$, using \eqref{eq-9}, 
\begin{multline*}
|h(\theta,\rho)-h(\theta',\rho')|
\leq C|z(\theta,\rho)-z(\theta',\rho')|^\gamma|\nabla \phi(z(\theta',\rho'))|\\
+|\nabla \phi(z(\theta,\rho))-\nabla \phi(\theta',\rho')|\cdot
|f(z(\theta,\rho))|\\
\leq C(|\theta-\theta'|^\gamma+|\rho-\rho'|^\gamma)
+C'|\rho'-\rho|^\gamma.
\end{multline*}
This proves that $h$ is a $\Lip(\gamma)$-vector field from $\Omega$ to $L(\RR^m,\RR^d)$.
In addition, 
since $|h(z)|\leq A+K|z|$ and $\nabla\phi(z)\leq C/|z|$, 
we get that $h$ is bounded.
\end{proof}

\begin{lemma}
\label{lem-2}
If $f$ is a $\Lip(1+\gamma)$-vector field from $\RR^d$ to $L(\RR^m,\RR^d)$ with $\gamma\in(0,1]$,
then~$h$ is a bounded $\Lip(1+\gamma)$-vector field from $\RR^m$ to $L(\RR^m,\RR^d)$.
\end{lemma}

\begin{proof}
Clearly, $h$ is differentiable and we have seen that it is bounded.
We have
\begin{equation*}
\nabla h(\theta,\rho)=\nabla^2\phi(z(\theta,\rho))\nabla z(\theta,\rho)
f(\theta,\rho)+\nabla \phi(z(\theta,\rho))\nabla f(\theta,\rho).
\end{equation*}
It follows from Lemma~\ref{lem-1} that 
$\nabla \phi(z(\theta,\rho))\nabla f(\theta,\rho)$ is 
$\gamma$-Hölder continuous on~$\Omega$ and bounded.
For $\rho,\rho'\geq 1$, 
\begin{multline*}
|\nabla h(\theta,\rho)-\nabla h(\theta,\rho)|
\leq C|z(\theta,\rho)-z(\theta',\rho')|^{\gamma}|\nabla \phi(z(\theta',\rho'))|\\
+|\nabla \phi(z(\theta,\rho))-\nabla \phi(\theta',\rho')|\cdot
|\nabla f(z(\theta,\rho))|
+|\nabla^2\phi(z(\theta,\rho))\nabla z(\theta,\rho)f(z(\theta,\rho))\\
-\nabla^2\phi(z(\theta,\rho))\nabla z(\theta,\rho)f(z(\theta,\rho))|.
\end{multline*}
For $z_i(\theta,\rho)=\theta\exp(\rho)$,
\begin{equation*}
\frac{\partial z_i}{\partial \theta_j}=\delta_{i,j}\exp(\rho)
\text{ and }\frac{\partial z_i}{\partial \rho}=\theta_i\exp(\rho)=z_i(\theta,\rho).
\end{equation*}
Thus, $\nabla z(\theta,\rho)$ also satisfies \eqref{eq-8} and \eqref{eq-9}
and then
\begin{equation*}
|\nabla z(\theta,\rho)-\nabla z(\theta',\rho')|
\cdot |f(z(\theta',\rho'))\nabla^2\phi(z(\theta',\rho'))|
\leq C|\rho'-\rho|^{\gamma}
+C'|\theta'-\theta|^{\gamma}.
\end{equation*}
Since $f$ is a Lipschitz function, if $\rho'\geq \rho\geq 1$,
\begin{multline*}
|(h(z(\theta,\rho))
-h(z(\theta',\rho')))|\cdot
|\nabla^2\phi(z(\theta',\rho'))\nabla z(\theta,\rho)|\\
\leq K|z(\theta,\rho)-z(\theta',\rho')|\frac{C\exp(\rho)}{\exp(2\rho')}
\leq C|\theta-\theta'|^{\gamma}+C'|\rho-\rho'|^{\gamma}.
\end{multline*}
Finally, for some positive constants $A$ and $B$, 
\begin{multline}
\label{eq-13}
|\nabla^2\phi(z(\theta',\rho'))-\nabla^2\phi(z(\theta,\rho))|
|f(z(\theta,\rho))\nabla z(\theta,\rho)|\\
\leq (A+B\exp(\rho))\exp(\rho)|\nabla^2\phi(z(\theta',\rho'))-\nabla^2\phi(z(\theta,\rho))|.
\end{multline}
If $|z'|\geq |z|$, 
\begin{multline*}
|\nabla^2\phi(z)-\nabla^2\phi(z')|
\leq C\left|\int_0^1 \frac{z'-z}{(z+\tau(z'-z))^2}d\tau
\right|\\
\leq \frac{C}{z^2}-\frac{C'}{(z')^2}
\leq |z'-z|\frac{|z|+|z'|}{|z|^2|z'|^2}
\leq \frac{2|z'-z|}{|z|^2}.
\end{multline*}
Then 
\begin{multline}
|\nabla^2\phi(z(\theta',\rho'))-\nabla^2\phi(z(\theta,\rho))|
|f(z(\theta,\rho))\nabla z(\theta,\rho)|\\
\leq 
C\frac{(A+B\exp(\rho))\exp(\rho)}{\exp(2\rho)}
(|\theta-\theta'|^{\gamma}+|\rho-\rho'|^{\gamma})
\end{multline}
which implies that the difference is $\gamma$-Hölder
when $\rho'\geq \rho\geq 1$.

Summarizing all these inequalities, we get that 
$\nabla h$ is $\gamma$-Hölder on $\Omega$.
This proves that $h$ is a $\Lip(1+\gamma)$-vector field from $\Omega$ to $L(\RR^d,\RR^m)$-vector field.
In addition, $\nabla h$ and $h$ are clearly bounded.
\end{proof}

To conclude, let us recall an important result from H.~Whitney, 
which is of course also valid for more regular vector fields.

\begin{theorem}[Whitney extension theorem~\cites{whitney,stein70a}]
\label{thm-1}
Let $\Omega$ be a closed subset of~$\RR^d$ and $\gamma\in(0,1]$.

There exists a linear operator $\fE$
for the space of bounded 
$\Lip(\gamma)$-vector fields from $\Omega$ to a space $\RR^{d'}$
to bounded $\Lip(\gamma)$-vector fields from $\RR^d$ to $\RR^{d'}$ such that 
\begin{align*}
\fE(f)_{|\Omega}&=f\\
\|\fE(f)\|_{\Lip}&\leq c\|f\|_{\Lip}\text{ with }
\|f\|_{\Lip}=\max\{\|f\|_{\infty}, H_\gamma(f)\},
\end{align*}
where $c$ depends only on $\gamma$.

There exists a linear operator~$\fE$ from the space of 
bounded $\Lip(1+\gamma)$-vector fields from $\Omega$ to $\RR^{d'}$ to bounded
$\Lip(1+\gamma)$-vector field from $\RR^d$ to $\RR^{d'}$ such that 
\begin{align*}
\fE(f)_{|\Omega}&=f\text{ and }\nabla \fE(f)_{|\Omega}=\nabla f,\\
\|\fE(f)\|_{\Lip}&\leq c\|f\|_{\Lip}
\text{ with }\|f\|_{\Lip}=\max\{\|f\|_{\infty}, \|\nabla f\|_{\infty}, H_\gamma(\nabla f)\},
\end{align*}
where $c$ depends only on $\gamma$.
\end{theorem}

This way, it is possible to extend $h$ to $\RR^d$. 

\begin{corollary}
\label{cor-1}
For a $f\in \Lip(1+\gamma)$-vector field from $\RR^d$ to $L(\RR^m,\RR^d)$,
the bounded $\Lip(1+\gamma)$-vector field $h$ from  $\Omega$ to $L(\RR^m,\RR^d)$ field $h$ defined by \eqref{eq-30} may be extended
to a bounded $\Lip(1+\gamma)$-vector field from $\RR^d$ to $L(\RR^d,\RR^m)$.
\end{corollary}

%%%%%%%%%%%%%%%%%%%%%%%%%%%%%%%%%%%%%%%%%%%%%%%%%%%%%%%%%%%%%%%%%%%%%%

\section{Global existence of a vector field with linear growth}
\label{sec:global}
\subsection{Case of a geometric rough path}

Let us now consider 
\begin{equation}
\label{eq-20}
\by_t=a+\int_0^t f(y_s)\vd \bx_s
\end{equation}
for a geometric rough path $\bx$ with values in $T(\RR^m)$ of finite $p$-variation 
controlled by $\omega$, and a $\Lip(1+\gamma)$ from $\RR^d$ to $L(\RR^m,\RR^d)$,
$2+\gamma>p$.

\begin{proposition}
\label{prop-2}
The solutions of \eqref{eq-20} are defined up to any time $T$ and there exists constants $C$
and $C'$ depending only on $|y_0|$, $\|f\|_{\infty}$, $\|\nabla f\|_{\infty}$, $H_\gamma(f)$, $\gamma$ 
and $p$ such that any solution $y$
to \eqref{eq-20} satisfies $|y_t|\leq C\exp(C'\|x\|_{p,\omega}^p\omega(0,T))$, 
\end{proposition}

\begin{proof}
Fix $T>0$.
With $\phi$ as above,
set $\psi(y)=\phi(b+y)$, where $b$ is chosen so that 
$\min_{t\in[0,T]}|b+y_t|\geq 1$.

We set $z_t=\psi(y_t)$ so that $z$ takes its values
in the set $\Omega=\SSS^d\times [1,+\infty)$.
Let us remark that $\nabla \psi\circ \psi^{-1}=\nabla\phi\circ \phi^{-1}$
and $f\circ \psi^{-1}=f(\phi^{-1}(\theta,\rho)-b)$.
The vector field $f(\cdot-b)$ is also a $\Lip(1+\gamma)$ from $\RR^d$ to $L(\RR^m,\RR^d)$-vector field.

It follows from Lemma~\ref{lem-2} and Corollary~\ref{cor-1} that 
\begin{equation*}
h(\theta,\rho)=\nabla\psi\circ \psi^{-1}(\theta,\rho)
f\circ\psi^{-1}(\theta,\rho)
\end{equation*}
may be extended to a bounded $\Lip(1+\gamma)$-vector field from $\RR^d$ to $L(\RR^m,\RR^d)$
still denoted by $h$.

From Proposition~\ref{prop-1}, it follows that 
$\by_t=a+\int_0^t f(y_s)\vd\bx_s$ has indeed a global 
solution.
In addition, 
$|z_t-z_0|\leq \mu+L'\|\bx\|_{p,\omega}^p\omega(0,T)$.
This implies that
\begin{equation*}
\left|\log\frac{|y_t+b|}{|a+b|}\right|\leq \mu+L'\|\bx\|^p_{p,\omega}\omega(0,T).
\end{equation*}
From this, we deduce that 
\begin{equation*}
|y_t|\leq (|a|+|b|-1)e^{\mu+L'\|\bx\|_{p,\omega}^p\omega(0,T)},\ t\in[0,T].
\end{equation*}
This proves that no explosion occurs at time $T$.
\end{proof}

%%%%%%%%%%%%%%%%%%%%%%%%%%%%%%%%%%%%%%%%%%%%%%%%%%%%%%%%%%%%%%%%%%%%%%
\subsection{Case of non-geometric rough paths}

We now consider that \eqref{eq-20} is driven by a non-geometric
rough path $\bx$ with values in $T(\RR^m)$ and of $p$-variation controlled by $\omega$. 
This rough path may be decomposed as the sum of 
a geometric rough path $\widehat{\bx}$ and a path $\beta$ of finite
$p/2$-variation value values in $\RR^m\otimes\RR^m$.

Using $\psi$ and the change of variable formula \eqref{eq-21}, 
we get that $z_t=\psi(y_t)$ is solution to 
\begin{multline}
\label{eq-22}
\bz_t=z_0+\int_0^t \nabla\psi\circ\psi^{-1}(z_s)f\circ\psi^{-1}(z_s)\vd\widehat{\bx}_s\\
+\int_0^t \nabla \psi\circ \psi^{-1}(z_s)
(f\cdot\nabla f)(\psi^{-1}(z_s)) d\beta_s.
\end{multline}

We have seen in \cite{lejay-victoir06a}
that there exists a solution to \eqref{eq-22}, 
provided that $h_1(z)=\nabla\psi\circ\psi^{-1}(z)f\circ\psi^{-1}(z)$ may be extended to  
a bounded $\Lip(1+\gamma)$-vector field from $\RR^d$ to $L(\RR^d,\RR^m)$, and 
$h_2(z)=\nabla \psi\circ \psi^{-1}(z)
(f\cdot\nabla f)\circ\psi^{-1}(z)$
may be extended to a bounded $\Lip(\gamma)$-vector field from $\RR^d$ to $L(\RR^m\otimes\RR^m,\RR^d)$.

Yet if $h_1(z)$ may be extended to a bounded $\Lip(1+\gamma)$-vector field from $\RR^d$ to $L(\RR^m,\RR^d)$
thanks to Lemma~\ref{lem-2} and Corollary~\ref{cor-1},
nothing ensures that $f\cdot\nabla f$ is $\gamma$-Hölder continuous
and then that $h_2$ may be extended to a bounded $\Lip(\gamma)$-vector field
from $\RR^d$ to $L(\RR^m\otimes\RR^m),\RR^d)$.

\begin{proposition}
For a $\Lip(1+\gamma)$-vector field $f$ from $\RR^d$ to $L(\RR^m,\RR^d)$-vector field $f$
such that $f\cdot\nabla f$ is a $\Lip(\gamma)$-vector field from $\RR^d$
to $L(\RR^m\otimes\RR^m,\RR^d)$-vector field
with $2+\gamma>p$, then there exists a global solution to \eqref{eq-20}.
\end{proposition}

\begin{proof}
Considering a solution $\bz$ to 
\begin{equation*}
\bz_t=z_0+\int_0^t h_1(z_s)\vd \bx_s+\int_0^t h_2(z_s)\vd\beta_s.
\end{equation*}
Then, for some constants $C_1$ and $C_2$ that depend on 
$\|h_1\|_{\Lip}$, 
$\|h_2\|_{\infty}$ and $H_\gamma(h_2)$, 
we know from the very
construction of a rough integral that for any $t\in[0,T]$, 
\begin{align*}
&\left|
z_t-z_0-h_1(z_0)\bx^1_{0,t}-\nabla h_1(z_0)\int_0^t \vd z_s\otimes\dd x_s
-h_2(z_0)\beta_{0,t}
\right|\leq C_1\omega(0,t)^{\theta},\\
&\left|
\int_0^t \vd z_s\otimes\dd x_s-h_1(z_0)\otimes 1\cdot \bx^2_{0,t}
\right|\leq C_2\omega(0,t)^\theta
\end{align*}
where $\theta>1$ and $\int \vd z_s\vd x_s$ denotes
the cross-iterated integral between $z$ and~$x$.
It follows that one can get an estimate of the type
\begin{equation*}
|z_t-z_0|\leq C(\omega(0,t)^\theta+\omega(0,t)^{1/p})\|\bx\|_{p,\omega}.
\end{equation*}
This estimate is less satisfactory that \eqref{eq-18}.
However, at the price of cumbersome computations, one should 
be able also to extend the results of \cite{lejay09a}
to deal with $(p,q)$-rough paths.

Now, the proof is similar to the one of Proposition~\ref{prop-1}.
\end{proof}

It is easy to construct a counter-example that show an explosion may occurs if we only require  linear growth of the vector fields.
\begin{example}
Consider the solution $y$ of the RDE in $\RR^2$,  $y_t=a+\int_0^t f(y_s)\vd x_s$  driven by the non-geometric  rough path
 $x_t=(1,0,(1\otimes 1)t)$  taking values in $1\oplus \RR \oplus (\RR \otimes \RR)$. This rough path lies above  the constant path at $0\in\RR$ and has only a pure area part which is symmetric and proportional to $t$.
Then by the above considerations~$y$ is also a solution to 
$y_t=a+\int_0^t (f \cdot \nabla f)(y_s)\vd s$.
Take the vector field $f\in \RR^2\to L(\RR,\RR^2)$ given by
$$
f(\xi)=(\sin(\xi_2) \xi_1,\xi_1 ), \ \xi=(\xi_1,\xi_2)\in\RR^2
$$
which is of linear growth and for which
$$
(f \cdot \nabla f)(\xi) = (\sin^2(\xi_2) \xi_1+\xi_1^2 \cos(\xi_2), \sin(\xi_2) \xi_1) .
$$
Take the initial point $a=(a_1,0)$ with $a_1>0$. Then $(y_t)_2=0$ and 
$
(y_t)_1 = a_1 +\int_0^t (y_s)_1^2 \vd s
$
so that $(y_t)_1\to+\infty$ in finite time. This proves that explosion may
occur in a finite time.
\end{example}

\paragraph*{Acknowledgement.} This article was written while the second author
was staying at \=Osaka University in a sabbatical leave and he whishes to thank prof. Aida
for his nice welcome and interesting discussions on the topic.

%%%%%%%%%%%%%%%%%%%%%%%%%%%%%%%%%%%%%%%%%%%%%%%%%%%%%%%%%%%%%%%%%%%%%%
%%%%%%%%%%%%%%%%%%%%%%%%%%%%%%%%%%%%%%%%%%%%%%%%%%%%%%%%%%%%%%%%%%%%%%
%%%%%%%%%%%%%%%%%%%%%%%%%%%%%%%%%%%%%%%%%%%%%%%%%%%%%%%%%%%%%%%%%%%%%%
%%%%%%%%%%%%%%%%%%%%%%%%%%%%%%%%%%%%%%%%%%%%%%%%%%%%%%%%%%%%%%%%%%%%%%
%%%%%%%%%%%%%%%%%%%%%%%%%%%%%%%%%%%%%%%%%%%%%%%%%%%%%%%%%%%%%%%%%%%%%%

\end{document}